\documentclass{amsart}
\usepackage{wrapfig}
\usepackage{graphicx} 
\usepackage{multicol} 
\usepackage{upgreek}
\usepackage{tikz}
\usepackage{mathrsfs}

\usepackage[latin1]{inputenc}
\usepackage[T1]{fontenc}
\usepackage{verbatim}
\usepackage{multirow}
\usepackage{geometry}
\usepackage{amssymb}
\usepackage{amsmath}
\usepackage{graphicx}
\usepackage{amsthm}
\usepackage{caption}
\usepackage{color}
\usepackage{enumerate}
\usepackage{graphicx}
\usepackage{float}
\usepackage{subfigure}

\newcommand{\eChar}{\begin{enumerate}[(i)]}
\newcommand{\eCharR}{\begin{enumerate}[(a)]}
\newcommand{\eBr}{\begin{enumerate}[(1)]}

\newcommand{\Abstract}

\title
{
Idleness Functions for Ollivier--Ricci Curvature on Hypergraphs
}

\author{Qing Xia}
\address{
School of Mathematical Sciences\\
University of Science and Technology of China\\
96 Jinzhai Road\\
Hefei 230026\\
Anhui Province\\
China}
\email{xq0420@mail.ustc.edu.cn}

\date{\today}
\theoremstyle{plain}
\newtheorem{lemma}{Lemma}[section]
\newtheorem{theorem}[lemma]{Theorem}

\newtheorem{corollary}[lemma]{Corollary}
\theoremstyle{definition}

\newtheorem{claim}[lemma]{Claim}

\newtheorem{definition}[lemma]{Definition}

\newtheorem{example}[lemma]{Example}

\numberwithin{equation}{section}

\begin{document}

\pagestyle{plain}

\begin{abstract}Let $\mathcal H=(V,E)$ be a locally finite simple hypergraph, equip $V$ with the hyperpath metric, and consider the equal-edges random walk \cite{CoupetteEtAl2023}. For adjacent vertices $x$ and $y$, we prove that the idleness function $\alpha\mapsto\kappa_\alpha^{\mathcal H}(x,y)$ is piecewise affine and with no more than three affine pieces. A separate mass-balance argument gives linearity on $[1/2,1]$ for every locally finite simple hypergraph and, consequently, a limit-free expression for the Lin--Lu--Yau curvature. In the $r$-uniform linear case, the hypergraph walk agrees exactly with the simple random walk on its 2-section. This reduction transfers the sharp endpoint intervals of Bourne, Cushing, Liu, M\"unch, and Peyerimhoff \cite{BourneEtAl2018}.
\end{abstract}
\maketitle


\section{Introduction and statement of results}

Ollivier's coarse Ricci curvature measures the contraction of one-step probability distributions in the $1$-Wasserstein metric \cite{Ollivier2009}. Its graph-theoretic specialization has become a standard tool in discrete geometry, with the Lin--Lu--Yau curvature arising from the small-step, or equivalently high-idleness, regime \cite{LinLuYau2011}. A basic structural question is therefore how the curvature changes when the random walk is made increasingly lazy. For locally finite graphs, Bourne, Cushing, Liu, M\"unch, and Peyerimhoff proved that the corresponding idleness function is concave and piecewise linear with at most three pieces; they also determine two intervals of monotonicity, one contaaining $0$ and the other containing $1$ \cite{BourneEtAl2018}.

\begin{theorem}\cite{BourneEtAl2018}\label{theorem:piecewise linear property for graphs} Let $G=(V, E)$ be a locally finite graph. Let $x, y\in V$ with $x\sim y$. Then the function $\alpha\mapsto\kappa_{\alpha}(x, y)$ is concave and piecewise linear over $[0,1]$ with at most 3 linear parts. Furthermore $\kappa_{\alpha}(x, y)$ is linear on the intervals
$$\left[0,\frac{1}{\operatorname{lcm}\left(d_x, d_y\right)+1}\right]\text{ and }\left[\frac{1}{\max\left(d_x, d_y\right)+1}, 1\right].$$
Thus, if we have the further condition $d_{x}=d_{y}$, then $\kappa_{\alpha}(x, y)$ has at most two linear parts.
\end{theorem}

Hypergraphs encode higher-order interactions that cannot, in general, be represented faithfully by an unweighted graph. This has led to several inequivalent extensions of discrete curvature, including multi-marginal transport curvature \cite{AsoodehGaoEvans2018}, curvature for directed hypergraphs \cite{EidiJost2020}, nonlinear-Laplacian approaches \cite{IkedaEtAl2022}, and flexible random-walk frameworks \cite{CoupetteEtAl2023}. The present article concerns the vertex-pair Ollivier--Ricci curvature associated with the equal-edges (EE) random walk \cite{CoupetteEtAl2023}. In the unweighted undirected setting, this is the same random walk as that used by Tian and Zhao \cite{TianZhao2025}, who also established the concavity of the associated idleness function.

For $r$-uniform linear hypergraphs, we identify the hypergraph walk with the simple random walk on its 2-section (see the definition in Section \ref{sec:preliminary}). The linear intervals then follow directly, and with explicit attribution, from the graph theorem of \cite{BourneEtAl2018}.

\begin{theorem}\label{theorem:piecewise linear property for linear hypergraphs} Let $\mathcal H=(V,E)$ be a locally finite $r$-uniform linear hypergraph. Let $x, y\in V$ with $x\sim_{\mathcal H}y$. Then the function $\alpha\mapsto\kappa_{\alpha}^{\mathcal H}(x, y)$ is concave and piecewise linear over $[0,1]$ with at most 3 linear parts. Furthermore, we have $\kappa_{\alpha}^{\mathcal H}(x, y)$ is linear on the intervals $$\left[0,\frac{1}{\operatorname{lcm}\left(d_x^{\mathcal H}, d_y^{\mathcal H}\right)(r-1)+1}\right]\text{ and }\left[\frac{1}{\max\left(d_x^{\mathcal H}, d_y^{\mathcal H}\right)(r-1)+1}, 1\right].$$
Thus, if we have the further condition $d_{x}^{\mathcal H}=d_{y}^{\mathcal H}$, then $\kappa_{\alpha}^{\mathcal H}(x, y)$ has at most two linear parts.

\end{theorem}

\begin{example}\label{ex:nonuniform}
Below we present the hypergraph $\mathcal H_0$, together with its associated idleness function:
\begin{figure}[H]
    \begin{minipage}{6cm}
    \centering
    \includegraphics[height=3.7cm,width=3.8cm]{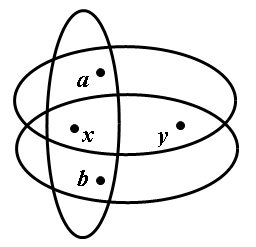}
    \caption{$\mathcal H_0$}
    \label{fig:H_0}
    \end{minipage}
    \begin{minipage}{9cm}
    \centering
    \includegraphics[height=4.2cm,width=5.7cm]{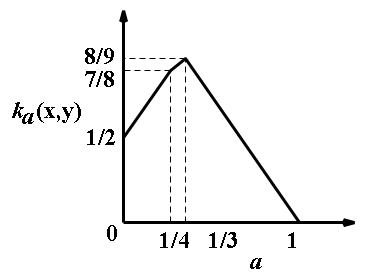}
    \caption{function $\alpha\mapsto\kappa_{\alpha}^{\mathcal H_0}(x, y)$}
    \label{fig:function}
    \end{minipage}
\end{figure}

Note that, under the assumption \(d_x^{\mathcal H} \ge d_y^{\mathcal H}\), the left endpoint of the linear interval containing \(\alpha = 1\) in Theorem~\ref{theorem:piecewise linear property for linear hypergraphs} is recovered by setting \(\mu_x^\alpha(y) = \mu_y^\alpha(y)\). For a general $r$-uniform hypergraph $\mathcal H=(V,E)$, this interval takes the form
\[
\left[
\frac{\sum_{\substack{e\in E\\e \ni \{x,y\}}} 1}
     {(d_x^{\mathcal H})(r-1) + \sum_{\substack{e\in E\\e \ni \{x,y\}}} 1},
\; 1
\right];
\]
for instance, the hypergraph \(\mathcal{H}_0\) in Figure~\ref{fig:H_0} yields \([1/4, 1]\). Yet the function depicted in Figure~\ref{fig:function} fails to be monotone on \([1/4, 1]\).

\end{example}
The passage from graphs to hypergraphs is not merely a change of notation. If the hyperedges incident to a vertex have different cardinalities, the resulting transition probabilities on neighboring vertices are nonuniform. Moreover, nonlinearity further undermines the uniformity of the probability distribution. Consequently, the general hypergraph walk is usually not the simple random walk on its 2-section, and graph formulas expressed only through graph degrees cannot be transferred directly. For a general locally finite hypergraph, we have:
\begin{theorem}\label{theorem:piecewise linear property for hypergraphs} Let $\mathcal H=(V,E)$ be a locally finite hypergraph. Let $x, y\in V$ with $x\sim_{\mathcal H} y$. Then the function $\alpha\mapsto\kappa_{\alpha}^{\mathcal H}(x, y)$ is concave and piecewise linear over $[0,1]$ with at most 3 linear parts. Furthermore, we have $\kappa_{\alpha}^{\mathcal H}(x, y)$ is linear on the interval $$\left[\frac{1}{2}, 1\right].$$ 
\end{theorem}
As an immediate consequence of Theorems \ref{theorem:piecewise linear property for linear hypergraphs} and \ref{theorem:piecewise linear property for hypergraphs}, we obtain the following Crollary:
\begin{corollary}
    Let $\mathcal H=(V, E)$ be a locally finite hypergraph. Let $x, y\in V$ with $x\sim_{\mathcal H} y$. Then we have
    \[\kappa_{LLY}^{\mathcal H}(x, y)=\frac{\kappa_\alpha^{\mathcal H}(x,y)}{1-\alpha},\ \alpha\in\left[\frac{1}{2}, 1\right].\]
If $\mathcal H$ is $r$-uniform linear, then 
    \[\kappa_{LLY}^{\mathcal H}(x, y)=\frac{\kappa_\alpha^{\mathcal H}(x,y)}{1-\alpha},\ \alpha\in\left[\frac{1}{\max\left(d_x^{\mathcal H}, d_y^{\mathcal H}\right)(r-1)+1}, 1\right].\]
\end{corollary}

The paper is arranged as follows. Section~\ref{sec:preliminary} defines the basic concepts and notation. Section~\ref{sec:transport} records the transport lemmas needed later.  Section~\ref{sec:proof} proves the main results.


\section{Preliminaries}\label{sec:preliminary}
Throughout, $\mathcal{H}=(V,E)$ is a locally finite, simple hypergraph.
Simplicity means that $E$ is an antichain under $\subseteq$ (i.e., for distinct
$e,f\in E$, neither $e\subseteq f$ nor $f\subseteq e$) and that $|e|\ge 2$
for every $e\in E$. Local finiteness means that each vertex belongs to only
finitely many hyperedges. Moreover, each hyperedge is finite. We write
\[
    d_x^{\mathcal H}:=\#\{e\in E:x\in e\}
\]
for the \emph{incidence degree} of $x$. Two distinct vertices are adjacent, written $x\sim_{\mathcal H}y$, if they lie in a common hyperedge. A \emph{hyperpath} joining $u,v\in V$ is a finite sequence
$\gamma=(e_1,\dots,e_l)$ of distinct hyperedges such that
$u\in e_1$, $v\in e_l$, and $e_j\cap e_{j+1}\neq\varnothing$
for $1\le j\le l-1$.
The \emph{distance} $d_{\mathcal H}(u,v)$ is the minimum number of hyperedges
in such a hyperpath, i.e.
\[
d_{\mathcal H}(u,v):=\inf_\gamma |\gamma|,
\]
where the infimum runs over all hyperpaths $\gamma$ from $u$ to $v$;
if none exists, set $d_{\mathcal H}(u,v)=\infty$, and $d_{\mathcal H}(u,u)=0$. For any vertex $x\in V$, let $B_1(x)=\{z\in V: d_{\mathcal H}(x,z)\leq1\}$.

The \emph{2-section} of $\mathcal{H}$, denoted by $[{\mathcal{H}}]_2$,
is the graph with vertex set $V$ in which two distinct vertices
$u,v\in V$ are adjacent if and only if $u\sim_{\mathcal H}v$. Thus, $[{\mathcal{H}}]_2$ is a simple graph. In $[{\mathcal{H}}]_2$, let $d_u$, $u\sim v$, and $d(u,v)$ denote the vertex degree, adjacency relation, and the graph distance, respectively. Note that $d(u,v)=d_{\mathcal{H}}(u,v)$ for any two vertices $u,v\in V$.

We first recall the \emph{Wasserstein distance} in the hypergraph.

\begin{definition}[\textbf{Wasserstein distance}]
Let $\mathcal H=(V,E)$ be a locally finite hypergraph, $\mu_1$ and $\mu_2$ be two finitely supported probability measures on $V$. Their $1$-Wasserstein distance is
\begin{equation}\label{eq:W_inf}
W_1^{\mathcal H}(\mu_1,\mu_2)
:=\inf_{\pi\in\Pi(\mu_1,\mu_2)}
  \sum_{u\in V}\sum_{v\in V} d_{\mathcal H}(u,v)\,\pi(u,v),
\end{equation}
where $\Pi(\mu_1,\mu_2)$ is the set of all maps
$\pi:V\times V\to[0,1]$ satisfying
\[
\left\{\begin{array}{l}\sum_{w\in V}\pi(u, w)=\mu_1(u),\\ \sum_{w\in V}\pi(w, v)=\mu_2(v),\end{array}\right.
\]
and any such $\pi$ is called a \emph{transport plan}. We call $\pi$ an optimal transport plan if $\pi$ attains the infimum in \eqref{eq:W_inf}.
\end{definition}

A key consequence of Kantorovich duality (\cite[Theorem~1.14]{Villani2003}) is an alternative formula for the Wasserstein distance:
\begin{equation}\label{eq:W_sup}
W_1^{\mathcal H}(\mu_1,\mu_2)
:=\sup_{f\in \emph{1-Lip(V)}} \sum_{u\in V}f(u)\left(\mu_1(u)-\mu_2(u)\right),
\end{equation}
where $f$ is \emph{1-Lip(V)} if $|f(x)-f
(y)|\leq d_{\mathcal H}(x,y)$ for all $x,y\in V$. We call $f$ an optimal Kantorovich potential transporting $\mu_1$ to $\mu_2$ if $f$ attains the supremum in \eqref{eq:W_sup}.

Given $x\in V$, we introduce a distinguished probability measure $\mu_x^\alpha$ with idleness parameter $\alpha\in [0,1]$:
\begin{equation}\label{eq:h_mu}
\mu_{x}^{\alpha}(z):=\left\{\begin{array}{ll}\alpha,&\text{ if } z=x;\\ \frac{1-\alpha}{d_x^{\mathcal H}}\sum_{\substack{e^{\prime}\in E\\e^{\prime}\ni\{x,z\}}}\frac{1}{\left|e^{\prime}\right|-1},&\text{ if } z\sim_{\mathcal H} x;\\ 0,&\text{ otherwise,}\end{array}\right.
\end{equation}
where $\left|e^{\prime}\right|$ denotes the cardinality of $e^{\prime}$.
\begin{definition}[\textbf{$\alpha$-Ollivier-Ricci curvature} and \textbf{Lin--Lu--Yau curvature} \cite{TianZhao2025}]\label{def:ollivier_and_lly}
Let $\mathcal H=(V,E)$ be a locally finite hypergraph. The \emph{$\alpha$-Ollivier-Ricci curvature} for any two vertices $x,y\in V$ is defined as
\begin{equation}\label{eq:ollivier}
    \kappa_{\alpha}^{\mathcal H}(x,y)
    :=1-\frac{W_1^{\mathcal H}(\mu_x^{\alpha},\mu_y^{\alpha})}{d_{\mathcal H}(x,y)}.
\end{equation}
The \emph{Lin--Lu--Yau curvature} of $x$ and $y$ is defined as
\begin{equation}\label{eq:lly}
    \kappa_{LLY}^{\mathcal H}(x,y)
    :=\lim_{\alpha\uparrow1}
      \frac{\kappa_{\alpha}^{\mathcal H}(x,y)}{1-\alpha}.
\end{equation}
\end{definition}

\section{Transport lemmas}\label{sec:transport}
In this section we record several lemmas needed in Section~\ref{sec:proof}.
These are optimal transport results on discrete metric spaces and are direct extensions of \cite[ Lemmas~3.1, 3.3, 4.1, 4.2]{BourneEtAl2018}.
Their proofs carry over unchanged to hypergraphs and are omitted for brevity.

\begin{lemma}[see Lemma 3.1 in \cite{BourneEtAl2018}]\label{lem:3.1}
Let $\mathcal H=(V, E)$ be a locally finite hypergraph. Let $x, y\in V$ with $x\sim_{\mathcal H} y$. Let $\alpha\in[0,1]$. Let $\pi$ and $f$ be an optimal transport plan and an optimal Kantorovich potential transporting $\mu_{x}^{\alpha}$ to $\mu_{y}^{\alpha}$, respectively. Let $u, v\in V$ with $\pi(u, v)\neq 0$. Then
$$f(u)-f(v)=d(u, v).$$
\end{lemma}

\begin{lemma}[see Lemma 3.3 in \cite{BourneEtAl2018}]\label{lem:3.3} Let $\mathcal H=(V, E)$ be a locally finite hypergraph. Let $x, y\in V$ with $x\sim_{\mathcal H} y$. Let $\alpha\in[0,1]$. Then there exists $f\in 1$-$Lip(V)$ such that
$$ W_1^{\mathcal H}\left(\mu_x^\alpha,\mu_y^{\alpha}\right)=\sum_{w\in V}f(w)\left(\mu_x^{\alpha}(w)-\mu_y^\alpha(w)\right),$$
and $f(w)\in \mathbb{Z}$ for all $w\in V$.
\end{lemma}

\begin{lemma}[see Lemma 4.1 in \cite{BourneEtAl2018}]\label{lem:4.1}
Let $\mathcal H=(V, E)$ be a locally finite hypergraph. Set $\mu_{1}$ and $\mu_{2}$ be probability measures on $V$. Then there exists an optimal transport plan $\pi$ transporting $\mu_{1}$ to $\mu_{2}$ with the following property: For all $x\in V$ with $\mu_{1}(x)\leq\mu_{2}(x)$ we have $\pi(x, x)=\mu_{1}(x)$.
\end{lemma}

\begin{lemma}[see Lemma 4.2 in \cite{BourneEtAl2018}]\label{lem:4.2}
Let $\mathcal H=(V, E)$ be a locally finite hypergraph. Let $x, y\in V$ with $x\sim_{\mathcal H} y$. Let $0\leq \alpha_{1}\leq \alpha_{2}\leq 1$. If there exists a $1$-$Lip(V)$ function $f$ which is an optimal Kantorovich potential transporting $\mu_{x}^{\alpha_1}$ to $\mu_{y}^{\alpha_2}$ and transporting $\mu_{x}^{\alpha_{2}}$ to $\mu_{y}^{\alpha_{2}}$, then $W_{x y}^{\mathcal H}:[0,1]\rightarrow\mathbb{R}$, $W_{x y}^{\mathcal H}(\alpha)=W_{1}^{\mathcal H}\left(\mu_{x}^{\alpha},\mu_{y}^{\alpha}\right)$, is linear on $\left[\alpha_{1}, \alpha_{2}\right]$.
\end{lemma}


\section{Proofs of the main results}\label{sec:proof}

\begin{proof}[Proof of Theorem \ref{theorem:piecewise linear property for linear hypergraphs}]
Let $\mathcal H=(V,E)$ be a locally finite $r$-uniform linear hypergraph. Let $[\mathcal H]_2=(V,E^\prime)$ be the 2-section of $\mathcal H$. Then $[\mathcal H]_2$ is simple, and we have 
\[
d(u,v)=d_{\mathcal H}(u,v),\ \forall u,v\in V. 
\]
For any $x\in V$, and $\alpha\in [0,1]$, we define the following probability measures in $[{\mathcal H}]_2$:
\[
\tilde{\mu}_{x}^{\alpha}(z):=\left\{\begin{array}{ll}\alpha,&\text{ if } z=x;\\ \frac{1-\alpha}{d_x},&\text{ if } z\sim x;\\ 0,&\text{ otherwise.}\end{array}\right.
\]
Since $\mathcal H$ is $r$-uniform linear, by \eqref{eq:h_mu},we have
\[
\mu_{x}^{\alpha}(z):=\left\{\begin{array}{ll}\alpha,&\text{ if } z=x;\\ \frac{1-\alpha}{d_x^{\mathcal H}(r-1)},&\text{ if } z\sim_{\mathcal H} x;\\ 0,&\text{ otherwise,}\end{array}\right.
\]
and
\begin{align*}
    d_x&=\#\{z\in V: z\sim_{\mathcal H} x\}\\ 
    &=d_x^{\mathcal H}(r-1).
\end{align*}
Thus, 
\[
\tilde{\mu}_{x}^{\alpha}(z)=\mu_{x}^{\alpha}(z),\ \forall z\in V.
\]
Therefore, for any $x,y\in V$ with $x\sim_{\mathcal H} y$, and $\alpha\in [0,1]$, we have 

\begin{align*}
    \kappa_\alpha^{\mathcal H}(x,y)&=1-W_1^{\mathcal H}(\mu_x^\alpha,\mu_y^\alpha)\\ 
    &=1-\inf_{\pi\in\Pi(\mu_x^\alpha,\mu_y^\alpha)}
  \sum_{u\in V}\sum_{v\in V} d_{\mathcal H}(u,v)\,\pi(u,v)\\
  &=1-\inf_{\pi\in\Pi(\tilde{\mu}_x^\alpha,\tilde{\mu}_y^\alpha)}
  \sum_{u\in V}\sum_{v\in V} d(u,v)\,\pi(u,v)\\
  &=1-W_1(\tilde{\mu}_x^\alpha,\tilde{\mu}_y^\alpha)\\
  &=\kappa_\alpha(x,y),
\end{align*}
where $W_1(\mu_x^\alpha,\mu_y^\alpha)$ denotes the 1-Wasserstein distance between $\mu_x^\alpha$ and $\mu_y^\alpha$ in $[\mathcal H]_2$, see \cite[Dedinition 2.1]{BourneEtAl2018}.

The proof is now complete by Theorem~\ref{theorem:piecewise linear property for graphs}.
\end{proof}

Now we are prepared for the proof of Theorem \ref{theorem:piecewise linear property for hypergraphs}.
\begin{proof}[Proof of Theorem \ref{theorem:piecewise linear property for hypergraphs}]
Let $\mathcal H=(V,E)$ be a locally finite hypergraph. Let $x, y\in V$ with $x\sim_{\mathcal H} y$. We first prove that the function $\alpha\mapsto\kappa_{\alpha}(x, y)$ is piecewise linear over $[0,1]$ with at most 3 linear parts. The following argument is adapted from the proof of Bourne et al. \cite[Theorem 3.4]{BourneEtAl2018}. The essential point is that their proof depends only on the affine dependence of the idleness measures, the existence of integer-valued optimal Kantorovich potentials, and the condition $d_{\mathcal H}(x,y)=1$, rather than on the binary structure of graph edges. Then we have 
\begin{align*}
W_{1}^{\mathcal H}\left(\mu_{x}^{\alpha},\mu_{y}^{\alpha}\right)&=\sup_{f\in 1\text{-Lip(V)}}\sum_{w\in V}f(w)\left(\mu_{x}^{\alpha}(w)-\mu_{y}^{\alpha}(w)\right)\\
&=\sup_{\substack{f\in 1\text{-Lip(V)}\\ f: V\rightarrow\mathbb{Z}\\ f(y)=0}}\sum_{w\in V}^{\mathcal H}f(w)\left(\mu_{x}^{\alpha}(w)-\mu_{y}^{\alpha}(w)\right)\ \quad\text{(by Lemma \ref{lem:3.3})}\\
&=\max_{j\in\{-1,0,1\}}\left\{j\left(\alpha-\frac{1-\alpha}{d_{y}^{\mathcal H}}\sum_{\substack{e\in E\\e\ni\{x,y\}}}\frac{1}{|e|-1}\right)+\frac{1-\alpha}{d_{x}^{\mathcal H} d_{y}^{\mathcal H}}c_j\right\}\  \quad\text{(since $f(x)\in\{-1,0,1\}$)}\\
&=\max\left\{f_{-1}(\alpha), f_{0}(\alpha), f_{1}(\alpha)\right\},
\end{align*}
where $c_j$ and $f_j$, $j\in\{-1,0,1\}$, are defined by\[\begin{cases}
    \mathcal{A}_j=&\{f: V\rightarrow\mathbb{Z}|f(x)=j,f(y)=0,f\in 1\text{-Lip(V)}\},\\
    c_{j}=&\sup_{f\in\mathcal{A}_{j}} d_y^{\mathcal H}\left(\sum_{\substack{z\sim x\\ z\neq y}}f(z)\sum_{\substack{e\in E\\e\ni\{x,z\}}}\frac{1}{|e|-1}\right)\\&-d_x^{\mathcal H}\left(\sum_{\substack{z\sim y\\ z\neq x}}f(z)\sum_{\substack{e\in E\\e\ni\{y,z\}}}\frac{1}{|e|-1}\right),\\
    f_j(\alpha)=&j\left(\alpha-\frac{1-\alpha}{d_y^{\mathcal H}}\sum_{\substack{e\in E\\e\ni\{x,y\}}}\frac{1}{|e|-1}\right)+\frac{1-\alpha}{d_x^{\mathcal H} d_y^{\mathcal H}} c_j.
\end{cases}\]
Therefore
$$\kappa_\alpha^{\mathcal H}(x, y)=1-\max\left\{f_{-1}(\alpha), f_{0}(\alpha), f_{1}(\alpha)\right\}.$$
Then $\alpha\longmapsto\kappa_\alpha^{\mathcal H}(x,y)$ is piecewise affine on $[0,1]$, with at most three affine pieces. 

Let $\pi$ be an optimal transport plan transporting $\mu_x^\alpha$ to $\mu_y^\alpha$ which satisfies the properties established in Lemma~\ref{lem:4.1}. Without loss of generality, Let $d_x^{\mathcal H}\geq d_y^{\mathcal H}$. We claim that:
\begin{claim}\label{cl:4.1}
If $f$ is an optimal Kantorovich potential transporting $\mu_x^{\alpha}$ to $\mu_y^{\alpha}$, then $f(x)-f(y)=1$ for all $\alpha>\frac{1}{2}$. 
\end{claim}
\begin{proof}
Since,
\[
\mu_x^\alpha(y)=\frac{1-\alpha}{d_x}\sum_{\substack{e\in E\\ e\ni\{x,y\}}}\frac{1}{|e|-1}\leq 1-\alpha<\alpha=\mu_y^\alpha(y),
\]
so there exists $z\in B_1(x)\setminus\{y\}$ with $\pi(z,y)>0$. If $z=x$, the claim follows from Lemma~\ref{lem:3.1}. If $z\neq x$,  observe that
\[
\alpha-\frac{1-\alpha}{d_y}\sum_{\substack{e\in E\\ e\ni\{x,y\}}}\frac{1}{|e|-1}=\sum_{v\in V\setminus\{x,y\}}\pi(x,v)\leq\sum_{v\in V\setminus\{x,y\}}\mu_y^\alpha(v)=1-\alpha-\frac{1-\alpha}{d_y}\sum_{\substack{e\in E\\ e\ni\{x,y\}}}\frac{1}{|e|-1}.
\]
This yields $\alpha\leq\frac{1}{2}$, contradicting the hypothesis. Hence $f(x)-f(y)=1$ for all $\alpha>\frac{1}{2}$.
\end{proof}
 Let $\frac{1}{2}<\alpha_0<1$, and let $f_*$ be an optimal Kantorovich potential transporting $\mu_x^{\alpha_0}$ to $\mu_y^{\alpha_0}$.
Then, by Claim \ref{cl:4.1}, we have $f_*(x)-f_*(y)=1$. Note that any $1$-$Lip(V)$ function $g$ satisfying $g(x)-g(y)=1$ is an optimal Kantorovich potential transporting $\mu_x^1$ to $\mu_y^1$. Hence, by Lemma~\ref{lem:4.2}, $\alpha\mapsto\kappa_\alpha(x,y)$ is linear on $[\alpha_0,1]$. By continuity, this linearity persists on $[\tfrac{1}{2},1]$.
This completes the proof of Theorem~\ref{theorem:piecewise linear property for hypergraphs}.
\end{proof}


\section*{Acknowledgements}
During the preparation of this work, the author used ChatGPT (Plus subscription) for language polishing in the Abstract and Introduction, and for assistance in constructing Example~\ref{ex:nonuniform}. The author critically reviewed and verified all AI-generated content and takes full responsibility for the integrity and accuracy of this publication.


\end{document}